\documentclass[12pt]{amsart}
\usepackage{amsmath, amsfonts, amssymb, amsthm,hyperref,mathtools,array}
\hypersetup{hypertex=true,
	colorlinks=true,
	linkcolor=blue,
	anchorcolor=blue,
	citecolor=blue}
\usepackage{bm}
\allowdisplaybreaks[4]
\def\({\bg(}
\def\){\bg)}

\def\Gal{{\rm Gal}}

\def\v{{\bm v}}

\def\0{{\bm 0}}

\def\alg{{\rm alg}}

\def\Re{{\rm Re}}

\def\id{{\rm id}}
\def\Im{{\rm Im}}
\def\pmod #1{\ ({\rm{mod}}\ #1)}
\def\mod #1{\ {\rm mod}\ #1}
\def\Ack{\medskip\noindent {\bf Acknowledgments}}

\theoremstyle{plain}
\newtheorem{theorem}{Theorem}[section]
\newtheorem{lemma}{Lemma}
\newtheorem{corollary}{Corollary}

\theoremstyle{definition}

\theoremstyle{remark}
\newtheorem{remark}{Remark}

\makeatletter
\@namedef{subjclassname@2020}{%
	\textup{2020} Mathematics Subject Classification}
\makeatother
\vspace{4mm}

\begin{document}
	
	\title[Products of Jacobi sums and cyclotomic matrices]
	{Products involving the real parts of Jacobi sums and related cyclotomic matrices}
	\author[H.-L. Wu and X.-H. Ji]{Hai-Liang Wu* and Xiao-Han Ji}
	
	\address {(Hai-Liang Wu) School of Science, Nanjing University of Posts and Telecommunications, Nanjing 210023, People's Republic of China}
	\email{\tt whl.math@smail.nju.edu.cn}
	
	\address {(Xiao-Han Ji) School of Science, Nanjing University of Posts and Telecommunications, Nanjing 210023, People's Republic of China}
	\email{\tt jxhmath@163.com}

	\keywords{Jacobi sums, cyclotomic matrices, finite fields.
		\newline \indent This research was supported by the Natural Science Foundation of China (Grant No. 12101321) and the Natural Science Foundation of the Higher Education Institutions of Jiangsu Province (Grant No. 25KJB110010).
		\newline \indent *Corresponding author.}
	
	\begin{abstract}
		Let $q$ be an odd prime power and $\chi_q$ be a generator of the group of all multiplicative characters of $\mathbb{F}_q$. In this paper, we study the arithmetic properties of the product
		$$R_q(\chi_q)=\prod_{0<k<(q-1)/4}\left(J_q(\phi_q,\chi_q^k)+J_q(\phi_q,\chi_q^{-k})\right),$$
		which is related to the real parts of Jacobi sums. Also, we reveal the connection between $R_q$ and the cyclotomic matrix 
		$$\left[\phi_q(s_i+s_j)\right]_{1\le i,j\le (q-1)/2},$$
		where $\phi_q$ is the unique quadratic multiplicative character of $\mathbb{F}_q$, and $s_1,s_2,\cdots,s_{(q-1)/2}$ are exactly all non-zero squares over $\mathbb{F}_q$. 
	\end{abstract}
	\maketitle
	
	\tableofcontents

	\section{Introduction}
	\setcounter{lemma}{0}
	\setcounter{theorem}{0}
	\setcounter{equation}{0}
	\setcounter{conjecture}{0}
	\setcounter{remark}{0}
	\setcounter{corollary}{0}
	
	\subsection{Notation} Throughout this paper, $p$ always denotes an odd prime. Let $\mathbb{Q}_p$ be the $p$-adic number field and $\mathbb{Q}_p^{\alg}$ be an algebraic closure of $\mathbb{Q}_p$. The ring of all $p$-adic integers is denoted by $\mathbb{Z}_p$. 
	
	Let $q=p^f$ be an odd prime power with $f\in\mathbb{Z}^+$. Let $\mathbb{F}_q$ be the finite field with $q$ elements and $\mathbb{F}_q^{\alg}$ be an algebraic closure of $\mathbb{F}_q$. Also, $\mathbb{F}_q^{\times}=\mathbb{F}_q\setminus\{0\}$ indicates the multiplicative cyclic group of all non-zero elements over $\mathbb{F}_q$. 
	
	The cyclic group of all multiplicative characters of $\mathbb{F}_q$ is denoted by $\widehat{\mathbb{F}_q^{\times}}$, and let $\chi_q$ be a generator of $\widehat{\mathbb{F}_q^{\times}}$. Given any multiplicative character 
	$$\psi: \mathbb{F}_q^{\times}\rightarrow\mathbb{C}^{\times}\ (\text{or}\  \mathbb{Q}_p^{\alg\times}),$$
	we additionally define $\psi(0) = 0$. Also, let $\phi_q$ be the unique quadratic multiplicative character of $\mathbb{F}_q$, that is, 
	$$\phi_q(x)=\begin{cases}
		1   &  \mbox{if}\ x\in\mathcal{S}_q,\\
		0  &  \mbox{if}\ x=0,\\
		-1 &  \mbox{otherwise},
	\end{cases}$$
	where 
	$$\mathcal{S}_q=\left\{s_1,s_2,\cdots,s_{(q-1)/2}\right\}=\left\{x^2:\ x\in\mathbb{F}_q^{\times}\right\}$$
	is the group of all non-zero squares over $\mathbb{F}_q$. For any $\chi,\psi\in\widehat{\mathbb{F}_q^{\times}}$, the Jacobi sum of $\chi$ and $\psi$ is defined by 
	$$J_q(\chi,\psi)=\sum_{x\in\mathbb{F}_q}\chi(x)\psi(1-x)\in\mathbb{Q}(\zeta_{q-1}),$$
	where $\zeta_{q-1}\in\mathbb{C}$ (or $\mathbb{Q}_p^{\alg}$) is a primitive $(q-1)$-th root of unity. 
	
	In addition, for any square matrix $M$ over a field, $\det M$ indicates the determinant of $M$. 
	
	\subsection{Background and motivation} There is a rich and long history concerning the study of Jacobi sums. For example, when $p\equiv 1\pmod 4$ is a prime, it is known that (see \cite[Theorem 6.2.9]{BEW}) the prime $p$ can be written as a sum of two integer squares, that is, 
	$$p=\left(\Re(J_p(\phi_p,\chi_p^{(p-1)/4})\right)^2+\left(\Im(J_p(\phi_p,\chi_p^{(p-1)/4})\right)^2,$$
	where $\Re(J_p(\phi_p,\chi_p^{(p-1)/4})$ and $\Im(J_p(\phi_p,\chi_p^{(p-1)/4})\in\mathbb{Z}$ are the real part and the imaginary part of $J_p(\phi_p,\chi_p^{(p-1)/4})$ respectively. 
	
	On the other hand, let $1<n<q-1$ be a divisor of $q-1$ with $q-1=kn$. Let
	$$C_n=\left\{[x,y,z]\in\mathbb{P}^2(\mathbb{F}_q^{\alg}):\ x^n+y^n=z^n\right\}$$
	be the projective Fermat curve defined over $\mathbb{F}_q$, where $\mathbb{P}^2(\mathbb{F}_q^{\alg})$ is the projective plane. Consider the zeta function of $C_n$ defined by 
	$$\zeta(t)=\exp\left(\sum_{r=1}^{+\infty}\frac{N(q^r)\cdot t^r}{r}\right),$$
	where $N(q^r)$ is the number of $\mathbb{F}_{q^r}$-rational points on $C_n$. As $C_m$ is a nonsingular absolutely irreducible curve of genus $(n-1)(n-2)/2$, by the Weil theorem, there is an integral polynomial $P_n(t)$ with $\deg (P_n(t))=(n-1)(n-2)$ such that 
	$$\zeta(t)=\frac{P_n(t)}{(1-t)(1-qt)}.$$
	Weil \cite{Weil} proved that 
	$$P_n(t)=\prod_{\substack{i,j\in[1,n-1]\\ i+j\not\equiv 0\pmod n}}\left(1+J_q(\chi_q^{ki},\chi_q^{kj})t\right)\in\mathbb{Z}[t],$$
	which is closely related to the products of Jacobi sums. Also, the local properties of Jacobi sums have been investigated extensively. For example, when $q=p$ is an odd prime, for any integers $1\le a,b\le p-2$ with $a+b<p-1$, by the classical Gross-Koblitz formula \cite{GK}, we have 
	$$J_p(\omega_p^{-a},\omega_p^{-b})=\frac{G_p(\omega_p^{-a})G_p(\omega_p^{-b})}{G_p(\omega_p^{-(a+b)})}=-\frac{\Gamma_p(a/(p-1))\Gamma_p(b/(p-1))}{\Gamma_p((a+b)/(p-1))},$$
	where $\omega_p$ is the Teich\"{u}muller character of $\mathbb{F}_p$, $G_p(\omega_p^{-a})$ is the Gauss sum over $\mathbb{F}_p$ and $\Gamma_p$ is the $p$-adic gamma function. 
	
	Next we introduce some works on cyclotomic matrices, which are closely related to Jacobi sums. Let $(\frac{\cdot}{p})$ be the Legendre symbol modulo $p$. Carlitz initiated the study of cyclotomic matrices. For instance, Carlitz \cite{Carlitz} showed that the determinant of the matrix 
	\begin{equation*}
		\det \left[\left(\frac{i-j}{p}\right)\right]_{1\le i,j\le p-1}=(-1)^{\frac{p-1}{2}}\cdot \prod_{k=1}^{p-1}J_p(\phi,\chi_p^k)=p^{\frac{p-3}{2}}.
	\end{equation*}
	As a generalization of the above, Wu and Wang \cite[Theorem 1.1]{Wu-Wang} proved that if $q\equiv 3\pmod 4$ is a prime power, then for any $1\le r\le q-2$, we have
	\begin{equation*}
		\det\left[\chi_q^r(s_i+s_j)+\chi_q^r(s_i-s_j)\right]_{1\le i,j\le (q-1)/2}=\prod_{k=0}^{(q-3)/2}J_q(\chi_q^r,\chi_q^{2k}).
	\end{equation*}
    
    The above results suggest that the product of Jacobi sums may be connected with the determinants of certain cyclotomic matrices. Motivated by the above works, in this paper, we concentrate on the product related to the real parts of Jacobi sums, that is, 
	$$R_q=\prod_{0<k<(q-1)/4}\left(J_q(\phi_q,\chi_q^k)+J_q(\phi_q,\chi_q^{-k})\right).$$ 
	We will see later that $R_q$ is closely related to the determinant of the matrix 
	$$A_q=\left[\phi_q(s_i+s_j)\right]_{1\le i,j\le (q-1)/2}.$$
	
	When $q=p$ is a prime, we see that 
	$$A_p=\left[\left(\frac{i^2+j^2}{p}\right)\right]_{1\le i,j\le (p-1)/2}.$$
	The arithmetic properties of $\det A_p$ have been studied extensively. Z.-W. Sun \cite[Theorem 1.2(iii)]{Sun} showed that 
	$$-\det A_p \mod{p\mathbb{Z}}\in\mathcal{S}_p.$$ 
	Furthermore, as an integer matrix, Z.-W. Sun conjectured that $-\det A_p=w_p^2$ for some integer $w_p$ whenever $p\equiv 3\pmod 4$. This conjecture was later confirmed by Alekseyev and Krachun. For the case $p\equiv 1\pmod 4$, if
	we write $p =c_p^2+4d_p^2$ with $c_p, d_p\in\mathbb{Z}$ and $c_p\equiv1\pmod 4$, then Cohen, Sun and Vsemirnov
	conjectured that (see \cite[Remark 4.2]{Sun}) $\det A_p/c_p=u_p^2$ for some $u_p\in\mathbb{Z}$. Wu \cite{Wucr} proved this conjecture later. 
	
	However, the above works have not provided exact descriptions of the integers $w_p$ and $u_p$. As an advance on this problem, we will determine the integers $w_p$ and $u_p$ completely in Theorem \ref{Thm. xq and det Aq} of this paper. 
	
	\subsection{Main Theorems} Given a real number $r$, define $\lfloor r \rfloor=\max\{k\in\mathbb{Z}:\ k\le r\}$. Now we state our first theorem.
	
	\begin{theorem}\label{Thm. Rq is an integer}
		Let $q=2n+1$ be an odd prime power and $\chi_q$ be a generator of $\widehat{\mathbb{F}_q^{\times}}$. Then $R_q(\chi_q)/2^{\lfloor n/2\rfloor}\in\mathbb{Z}$ and $R_q(\chi_q)$ is independent of the choice of the generator $\chi_q$ of $\widehat{\mathbb{F}_q^{\times}}$. 
	\end{theorem}
	
	\begin{remark}\label{Remark of Thm. Rq is an integer}
		As mentioned above, the product $R_q(\chi_q)\in\mathbb{Z}$ is independent of the choice of the generator $\chi_q$ of $\widehat{\mathbb{F}_q^{\times}}$. Thus, from now on, we abbreviate $R_q(\chi_q)$ as $R_q$ and let $x_q=R_q/2^{\lfloor n/2\rfloor}\in\mathbb{Z}$. 
	\end{remark}
	
	The values of $R_q$ for $7\le q\le 29$ are given in the following table.
	\begin{table}[h]
		\centering
		\begin{tabular}{|c|c|c|c|c|c|c|c|c|c|c|}
			\hline
			$q$     & 7    & 9 & 11 & 13 & 17 & 19 & 23 & 25 & 27 & 29 \\
			\hline
			$R_q$ & -4 & -2 & 16 & -12 & -60 & 256 & -1024 & 2400 & 4096 & 320 \\
			\hline
		\end{tabular}
	\end{table}
	
	To state our next theorem, we introduce the following notations. Let $E$ be a rational affine elliptic curve defined by $Y^2=X^3+X$. For any odd prime power $q=p^f$ with $p$ prime and $f\in\mathbb{Z}^+$, let 
	\begin{equation}\label{Eq. definition of trace of Frobenius}
		a_q(E)=-\sum_{x\in\mathbb{F}_q}\phi_q(x^3+x).
	\end{equation}
	
	If $E_p$ is the reduction of $E$ modulo $p$, then one can verify that 
	$$\# E_p(\mathbb{F}_q)=q+1-a_q(E),$$
	where 
	$$E_p(\mathbb{F}_q)=\left\{(x,y)\in\mathbb{F}_q:\ y^2=x^3+x\right\}\cup\{\infty\}$$
	is the set of all $\mathbb{F}_q$-rational points on $E_p$, and $\# S$ denotes the cardinality of a finite set $S$. 
	
	Now we state our second theorem, which reveals the connection between $\det A_q$ and $R_q$. 
	
	\begin{theorem}\label{Thm. xq and det Aq}
		Let $q=2n+1$ be an odd prime power and let $x_q=R_q/2^{\lfloor n/2\rfloor}\in\mathbb{Z}$.  
		
		{\rm (i)} Suppose $q\equiv 3\pmod 4$. Then 
		$$-\det A_q=x_q^2.$$
		
		{\rm (ii)} Suppose $q\equiv 1\pmod 4$. Then 
		$$2\cdot \det A_q=a_q(E)\cdot x_q^2.$$
		In particular, when $q=p\equiv 1\pmod 4$ is a prime, if we write $p=c_p^2+4d_p^2$ with $c_p,d_p\in\mathbb{Z}$ and $c_p\equiv 1\pmod 4$, then 
		$$\frac{\det A_p}{c_p}=x_p^2.$$
	\end{theorem}
	
	The next theorem concerns the local properties of the integer $R_q$. 
	
	\begin{theorem}\label{Thm. local properties of Rq}
		Let $q=p^f=2n+1$ be an odd prime power with $p$ prime and $f\in\mathbb{Z}^+$. Then 
		$$R_q\equiv 
		\begin{cases}
			0 \pmod{p}    &  \mbox{if}\ f>1\ \text{and}\ q\neq 9,\\
			-2\pmod{p}   &  \mbox{if}\  q=9,\\
			(-1)^{\lfloor n/2 \rfloor}\left(\frac{1}{4}\right)^{\lfloor n/2 \rfloor(1+\lfloor n/2 \rfloor)/2}\prod_{0<k<n/2}\binom{2k}{k} \pmod{p\mathbb{Z}_p}  &  \mbox{if}\ f=1.
		\end{cases}$$
		Moreover, when $f=1$, i.e., $q=p$ is an odd prime, we have 
		$$\left(\frac{R_p}{p}\right)=
		\begin{cases}
			(-1)^{(p-3)/4}  &  \mbox{if}\ p\equiv 3\pmod 4,\\
			1                           &  \mbox{if}\ p\equiv 1\pmod 4.
		\end{cases}.$$
	\end{theorem}
	
	As a direct consequence of the above theorems, when $p\equiv 3\pmod 4$, we can determine the explicit value of $R_p$. 
	
	\begin{corollary}\label{Corollary for Rp when p=3 mod 4}
		Let $p=2n+1\equiv 3\pmod 4$ be a prime. Then 
		$$R_p=\delta_p\cdot\sqrt{-2^{n-1}\det A_p},$$
		where $\delta_p\in\{\pm 1\}$ such that 
		$$\left(\frac{-2\delta_p\cdot\sqrt{-2^{n-1}\det A_p}}{p}\right)=1.$$
	\end{corollary}
	
	\subsection{Outline of this paper} We will prove Theorems \ref{Thm. Rq is an integer}--\ref{Thm. local properties of Rq} in Sections 1--3 respectively. 
	
	\section{Proof of Theorem \ref{Thm. Rq is an integer}}
	\setcounter{lemma}{0}
	\setcounter{theorem}{0}
	\setcounter{equation}{0}
	\setcounter{conjecture}{0}
	\setcounter{remark}{0}
	\setcounter{corollary}{0}
	
	Throughout this section, we always view the Jacobi sum $J_q(\chi_q^i,\chi_q^j)$ as an element of the global field $\mathbb{Q}(\zeta_{q-1})\subseteq \mathbb{C}$ for any $i,j\in\mathbb{Z}$. 
	
	We begin with the following result.
	
	\begin{lemma}\label{Lem. eigenvalues of Aq}
		Let $q=2n+1$ be an odd prime power. Then for any generator $\chi_q$ of $\widehat{\mathbb{F}_q^{\times}}$, the numbers $\lambda_1(\chi_q), \lambda_2(\chi_q), \cdots, \lambda_n(\chi_q)$ are exactly all the eigenvalues of $A_q$, where 
		$$\lambda_k(\chi_q)=\sum_{x\in\mathcal{S}_q}\phi_q(1+x)\chi_q^k(x)=\frac{(-1)^k}{2}\left(J_q(\phi_q,\chi_q^k)+J_q(\phi_q,\chi_q^{-k})\right)$$
		is an algebraic integer	for any $k\in[1,n]$. 
	\end{lemma}
	
	\begin{proof}
		For any integer $k\in [1,n]$, one can verify that 
		\begin{align*}
			\sum_{1\le j\le n}\phi_q(s_i+s_j)\chi_q^k(s_j)
			&=\sum_{1\le j\le n}\phi_q(1+s_j/s_i)\chi_q^k(s_j/s_i)\chi_q^k(s_i)\\
			&=\sum_{1\le j\le n}\phi_q(1+s_j)\chi_q^k(s_j)\chi_q^k(s_i)\\
			&=\lambda_k(\chi_q)\chi_q(s_i),
		\end{align*}
		where 
		$$\lambda_k(\chi_q)=\sum_{1\le j\le n}\phi_q(1+s_j)\chi_q^k(s_j)=\sum_{x\in\mathcal{S}_q}\phi_q(1+x)\chi_q^k(x)$$
		is an algebraic integer. This implies that 
		\begin{equation}\label{Eq. a in the proof of Lem. eigenvalues of Aq}
			A_q\v_k(\chi_q)=\lambda_k(\chi_q)\v_k(\chi_q)
		\end{equation}
		for any integer $k\in[1,n]$, where the column vector 
		$$\v_k(\chi_q)=\left(\chi_q^k(s_1), \chi_q^k(s_2), \cdots, \chi_q^k(s_n)\right)^T$$
		Observing that the vectors $\v_1(\chi_q), \v_2(\chi_q), \cdots, \v_n(\chi_q)$ are linearly independent, the numbers $\lambda_1(\chi_q), \lambda_2(\chi_q), \cdots, \lambda_n(\chi_q)$ are precisely all the eigenvalues of $A_q$ by (\ref{Eq. a in the proof of Lem. eigenvalues of Aq}). 
		
		Next we consider the number $\lambda_k(\chi_q)$. Let $\mathcal{N}_q=\mathbb{F}_q^{\times}\setminus\mathcal{S}_q$. For any integer $k\in[1,n]$, define the sum 
		$$n_k(\chi_q)=\sum_{x\in\mathcal{N}_q}\phi_q(1+x)\chi_q^k(x).$$
		Then one can verify that 
		\begin{align*}
			\sum_{x\in\mathcal{N}_q}\phi_q(1+x)\chi_q^{-k}(x)
			&=\sum_{x\in\mathcal{N}_q}\phi_q(x)\phi_q\left(1+\frac{1}{x}\right)\chi_q^k\left(\frac{1}{x}\right)\\
			&=-\sum_{x\in\mathcal{N}_q}\phi_q\left(1+\frac{1}{x}\right)\chi_q^k\left(\frac{1}{x}\right)\\
			&=-\sum_{x\in\mathcal{N}_q}\phi_q(1+x)\chi_q^k(x)\\
			&=-n_k(\chi_q).
		\end{align*}
		Applying this, for any integer $k\in[1,n]$, we immediately obtain 
		\begin{align*}
			2\lambda_k(\chi_q)
			&=\left(\lambda_k(\chi_q)+n_k(\chi_q)\right)+\left(\lambda_k(\chi_q)-n_k(\chi_q)\right)\\
			&=\sum_{x\in\mathbb{F}_q^{\times}}\phi_q(1+x)\chi_q^k(x)+\sum_{x\in\mathcal{S}_q}\phi_q(1+x)\chi_q^k(x)+\sum_{x\in\mathcal{N}_q}\phi_q(1+x)\chi_q^{-k}(x)\\
			&=\sum_{x\in\mathbb{F}_q^{\times}}\phi_q(1+x)\chi_q^k(x)+\sum_{x\in\mathcal{S}_q}\phi_q(x)\phi_q\left(1+\frac{1}{x}\right)\chi_q^k\left(\frac{1}{x}\right)+\sum_{x\in\mathcal{N}_q}\phi_q(1+x)\chi_q^{-k}(x)\\
			&=\sum_{x\in\mathbb{F}_q^{\times}}\phi_q(1+x)\chi_q^k(x)+\sum_{x\in\mathcal{S}_q}\phi_q(1+x)\chi_q^{-k}(x)+\sum_{x\in\mathcal{N}_q}\phi_q(1+x)\chi_q^{-k}(x)\\
			&=\sum_{x\in\mathbb{F}_q^{\times}}\phi_q(1+x)\chi_q^k(x)+\sum_{x\in\mathbb{F}_q^{\times}}\phi_q(1+x)\chi_q^{-k}(x)\\
			&=(-1)^k\sum_{x\in\mathbb{F}_q^{\times}}\phi_q(1+x)\chi_q^k(-x)+(-1)^k\sum_{x\in\mathbb{F}_q^{\times}}\phi_q(1+x)\chi_q^{-k}(-x)\\
			&=(-1)^k\cdot\left(J_q(\phi_q,\chi_q^k)+J_q(\phi_q,\chi_q^{-k})\right).
		\end{align*}
		
		In view of the above, we have completed the proof. 
	\end{proof}
	
	Let $m\ge1$ be an integer. For any integer $x$, let $\{x\}_m$ be the unique integer in the interval $[0,m)$ such that $\{x\}_m\equiv x\pmod{m\mathbb{Z}}$.

	As a generalization of the classical Gauss lemma, Jenkins \cite{Jenkins} proved the following result.
	
	\begin{lemma}\label{Lem. Gauss lemma for q=3 mod 4}
		Let $m$ be a positive odd integer and let $a\in\mathbb{Z}$ with $\gcd(a,m)=1$. Then 
		$$(-1)^{\#\left\{k\in(0, m/2):\ \{ak\}_m>m/2\right\}}=\left(\frac{a}{m}\right),$$
		where $(\frac{\cdot}{m})$ is the Jacobi symbol. Moreover, 
		$$\left(\frac{a}{m}\right)=(-1)^{K_a},$$
		where 
		$$K_a=\sum_{0<k<m/2}\left\lfloor\frac{2ak}{m}\right\rfloor.$$
	\end{lemma}
	
	Pan \cite[Lemma 1]{Pan} obtained the following result.
	
	\begin{lemma}\label{Lem. Pan}
		Let $m$ be a positive integer and let $a\in\mathbb{Z}$ with $\gcd(a,m)=1$. Then for any integers $1\le i<j\le m-1$, we have 
		$$\left\lfloor\frac{aj}{m}\right\rfloor-\left\lfloor\frac{ai}{m}\right\rfloor-\left\lfloor\frac{a(j-i)}{m}\right\rfloor=
		\begin{cases}
			1 & \mbox{if}\ \{a_i\}_m>\{a_j\}_m,\\
			0 & \mbox{otherwise}.
		\end{cases}$$
	\end{lemma}
	
	Using Lemmas \ref{Lem. Gauss lemma for q=3 mod 4}--\ref{Lem. Pan}, we can obtain the following result, which will play a key role in the proof of our first theorem.
	
	\begin{lemma}\label{Lem. key lemma for q=3 mod 4}
		Let $q=2n+1\equiv 3\pmod 4$ be an odd prime power. Then for any integer $a\in (0,q)$ with $\gcd(a,q-1)=1$, we have 
		$$\#\left\{k\in\left(0,\frac{n}{2}\right):\ \frac{n}{2}<\{ak\}_{2n}<\frac{3n}{2}\right\}\equiv 0\pmod{2\mathbb{Z}}.$$
	\end{lemma}
	
	\begin{proof}
		For $j=1,2,3,4$, define the set 
		$$X_j(a)=\left\{k\in\left(0,\frac{n}{2}\right):\ \frac{(j-1)}{2}n<\{ak\}_{2n}<\frac{j}{2}n\right\}.$$
		Clearly $X_i(a)\cap X_j(a)=\emptyset$ for any $1\le i<j\le 4$. 
		
		We first consider $X_2(a)\cup X_4(a)$. For any integers $k_1,k_2\in(0, n/2)$ with $k_1\neq k_2$, one can verify that 
		\begin{equation}\label{Eq. a in the proof of Lem. key lemma for q=3 mod 4}
			ak_1\not\equiv \pm ak_2\pmod{n\mathbb{Z}}.
		\end{equation}
		Note that 
		$$X_1(a)\cup X_3(a)=\left\{k\in\left(0,\frac{n}{2}\right):\ 0<\{ak\}_{2n}<\frac{1}{2}n\ \text{or}\ n<\{ak\}_{2n}<\frac{3}{2}n\right\},$$
		and 
		$$X_2(a)\cup X_4(a)=\left\{k\in\left(0,\frac{n}{2}\right):\ \frac{1}{2}n<\{ak\}_{2n}<n\ \text{or}\ \frac{3}{2}n<\{ak\}_{2n}<2n\right\}.$$
		Thus, by (\ref{Eq. a in the proof of Lem. key lemma for q=3 mod 4}) we obtain 
		$$X_2(a)\cup X_4(a)=\left\{k\in\left(0,\frac{n}{2}\right):\ \{ak\}_n>\frac{1}{2}n\right\}.$$
		By Lemma \ref{Lem. Gauss lemma for q=3 mod 4}, we have 
		\begin{equation}\label{Eq. b in the proof of Lem. key lemma for q=3 mod 4}
			(-1)^{\# X_2(a)\cup X_4(a)}=(-1)^{\# X_2(a)+\# X_4(a)}=\left(\frac{a}{n}\right).
		\end{equation}
		
		Next we consider $X_3(a)\cup X_4(a)$. Note that $\{an\}_{2n}=n$. Hence, by (\ref{Eq. a in the proof of Lem. key lemma for q=3 mod 4}) we obtain 
		$$X_3(a)\cup X_4(a)=\left\{k\in\left(0,\frac{n}{2}\right):\ \{ak\}_{2n}>\{an\}_{2n}=n\right\}.$$
		Since $a$ and $n$ are both odd, applying Lemma \ref{Lem. Pan}, one can verify that 
		\begin{align}\label{Eq. c in the proof of Lem. key lemma for q=3 mod 4}
			\# X_3(a)\cup X_4(a)
			&=\sum_{0<k<n/2}\left(\left\lfloor\frac{an}{2n}\right\rfloor-\left\lfloor\frac{ak}{2n}\right\rfloor-\left\lfloor\frac{a(n-k)}{2n}\right\rfloor\right)\notag\\
			&=\frac{(n-1)(a-1)}{4}-\sum_{0<k<n/2}\left(\left\lfloor\frac{ak}{2n}\right\rfloor+\left\lfloor\frac{a(n-k)}{2n}\right\rfloor\right)\notag\\
			&=\frac{(n-1)(a-1)}{4}-\sum_{1\le k\le n-1}\left\lfloor\frac{ak}{2n}\right\rfloor.
		\end{align}
		Note that 
		\begin{align*}
			\sum_{1\le k\le n-1}\left\lfloor\frac{ak}{2n}\right\rfloor+\sum_{1\le k\le (a-1)/2}\left\lfloor\frac{2nk}{a}\right\rfloor
			&=\#\left\{(x,y)\in\mathbb{Z}\times\mathbb{Z}:\ 0<x<n\ \text{and}\ 0<y<a/2\right\}\\
			&=\frac{(n-1)(a-1)}{2}.
		\end{align*}
		Combining this with (\ref{Eq. c in the proof of Lem. key lemma for q=3 mod 4}) and noting that $(n-1)(a-1)/2$ is even, we obtain 
		$$\# X_3(a)+\# X_4(a)\equiv \frac{(n-1)(a-1)}{4}+\sum_{1\le k\le (a-1)/2}\left\lfloor\frac{2nk}{a}\right\rfloor\pmod{2\mathbb{Z}}.$$
		Thus, using Lemma \ref{Lem. Gauss lemma for q=3 mod 4}, we have 
		\begin{equation}\label{Eq. d in the proof of Lem. key lemma for q=3 mod 4}
			(-1)^{\# X_3(a)+\# X_4(a)}=(-1)^{(n-1)(a-1)/4}\cdot\left(\frac{n}{a}\right).
		\end{equation}
		
		Assembling (\ref{Eq. b in the proof of Lem. key lemma for q=3 mod 4}) and (\ref{Eq. d in the proof of Lem. key lemma for q=3 mod 4}), we obtain 
		\begin{align*}
			(-1)^{\# X_2(a)+\# X_3(a)}
			&=(-1)^{\# X_2(a)+\# X_4(a)}\cdot (-1)^{\# X_3(a)+\# X_4(a)}\\
			&=\left(\frac{a}{n}\right)\cdot (-1)^{(n-1)(a-1)/4}\cdot\left(\frac{n}{a}\right)\\
			&=1.
		\end{align*}
		Thus, 
		$$\# X_2(a)\cup X_3(a)=\# X_2(a)+\# X_3(a) \equiv 0\pmod{2\mathbb{Z}}.$$
		
		In view of the above, we have completed the proof.
	\end{proof}
	
	We also need the following result (cf. \cite[Lemma 2.8]{Wu and Pan}).
	
	\begin{lemma}\label{Lem. transform of Jacobi sums}
		Let $q=2n+1$ be odd prime power and let $\chi_q$ be a generator of $\widehat{\mathbb{F}_q^{\times}}$. Then for any $k\in\mathbb{Z}$, we have 
		$$J_q(\phi_q,\chi_q^k)=\phi_q(-1)\cdot J_q(\phi_q,\phi_q\chi_q^{-k}).$$
	\end{lemma}
	
	Now we are in a position to prove our first theorem.
	
	{\noindent\bfseries Proof of Theorem \ref{Thm. Rq is an integer}}. Recall that $q=2n+1$ is an odd prime power and that $\chi_q$ is a generator of $\widehat{\mathbb{F}_q^{\times}}$. Let $\zeta_{q-1}$ be a primitive $(q-1)$-th root of unity. For any $a\in\mathbb{Z}$ with $\gcd(a,q-1)=1$, define the $\mathbb{Q}$-automorphism $\sigma_a$ of $\mathbb{Q}(\zeta_{q-1})$ by $\sigma_a(\zeta_{q-1})=\zeta_{q-1}^a$. It is known that 
	$$\Gal\left(\mathbb{Q}(\zeta_{q-1})/\mathbb{Q}\right)=\left\{\sigma_a:\ 1\le a\le q-1\ \text{and}\ \gcd(a,q-1)=1\right\}.$$
	
	Let $\sigma_a\in\Gal(\mathbb{Q}(\zeta_{q-1})/\mathbb{Q})$ be an arbitrary automorphism, where $1\le a\le q-1$ and $\gcd(a,q-1)=1$. Set 
	$$X_j(a)=\left\{k\in\left(0,\frac{n}{2}\right):\ \frac{(j-1)}{2}n<\{ak\}_{2n}<\frac{j}{2}n\right\}$$
	for $j=1,2,3,4$. 
	
	(i) We first consider the case $q\equiv 3\pmod 4$, i.e., $n\equiv 1\pmod 2$. 
	
	Note that $\chi_q^{\pm n}=\phi_q$. By Lemma \ref{Lem. transform of Jacobi sums} one can verify that 
	\begin{align}\label{Eq. a in the proof of global result for q=3 mod 4}
		&\prod_{k\in X_2(a)}\left(J_q(\phi_q,\chi_q^{ak})+J_q(\phi_q,\chi_q^{-ak})\right)\notag\\
		=&\prod_{k\in X_2(a)}\left(J_q(\phi_q,\chi_q^{\{ak\}_{2n}})+J_q(\phi_q,\chi_q^{-\{ak\}_{2n}})\right)\notag\\
		=&(-1)^{\# X_2(a)}\cdot \prod_{k\in X_2(a)}\left(J_q(\phi_q,\chi_q^{n-\{ak\}_{2n}})+J_q(\phi_q,\chi_q^{-n+\{ak\}_{2n}})\right).
	\end{align}
	
	For $X_3(a)$, applying Lemma \ref{Lem. transform of Jacobi sums}, we obtain 
	\begin{align}\label{Eq. b in the proof of global result for q=3 mod 4}
		&\prod_{k\in X_3(a)}\left(J_q(\phi_q,\chi_q^{ak})+J_q(\phi_q,\chi_q^{-ak})\right)\notag\\
		=&\prod_{k\in X_3(a)}\left(J_q(\phi_q,\chi_q^{\{ak\}_{2n}})+J_q(\phi_q,\chi_q^{-\{ak\}_{2n}})\right)\notag\\
		=&(-1)^{\# X_3(a)}\cdot \prod_{k\in X_3(a)}\left(J_q(\phi_q,\chi_q^{-n+\{ak\}_{2n}})+J_q(\phi_q,\chi_q^{n-\{ak\}_{2n}})\right).
	\end{align}
	
	Using Lemma \ref{Lem. transform of Jacobi sums} again, we see that 
	\begin{align}\label{Eq. c in the proof of global result for q=3 mod 4}
		&\prod_{k\in X_4(a)}\left(J_q(\phi_q,\chi_q^{ak})+J_q(\phi_q,\chi_q^{-ak})\right)\notag\\
		=&\prod_{k\in X_4(a)}\left(J_q(\phi_q,\chi_q^{\{ak\}_{2n}})+J_q(\phi_q,\chi_q^{-\{ak\}_{2n}})\right)\notag\\
		=&\prod_{k\in X_4(a)}\left(J_q(\phi_q,\chi_q^{2n-\{ak\}_{2n}})+J_q(\phi_q,\chi_q^{-2n+\{ak\}_{2n}})\right).
	\end{align}
	
	Set 
	\begin{align*}
		U_1(a)&=\left\{\{ak\}_{2n}:\ k\in X_1(a)\right\},\\
		U_2(a)&=\left\{n-\{ak\}_{2n}:\ k\in X_2(a)\right\},\\
		U_3(a)&=\left\{-n+\{ak\}_{2n}:\ k\in X_3(a)\right\},\\
		U_4(a)&=\left\{2n-\{ak\}_{2n}:\ k\in X_4(a)\right\}.
	\end{align*}
	Since $ak_1\not\equiv \pm ak_2 \pmod{n\mathbb{Z}}$ for any $0<k_1\neq k_2<n/2$, we see that $U_i\cap U_j=\emptyset$ for any  $1\le i<j\le 4$ and 
	$$\bigcup_{1\le i\le 4}U_i(a)=\left\{1,2,\cdots,\frac{n-1}{2}\right\}.$$
	From this and (\ref{Eq. a in the proof of global result for q=3 mod 4})--(\ref{Eq. c in the proof of global result for q=3 mod 4}), we obtain 
	\begin{align*}
		\sigma_a(R_q(\chi_q))
		&=\prod_{0<k<n/2}\left(J_q(\phi_q,\chi_q^{ak})+J_q(\phi_q,\chi_q^{-ak})\right)\\
		&=\prod_{1\le j\le 4}\prod_{k\in X_j(a)}\left(J_q(\phi_q,\chi_q^{\{ak\}_{2n}})+J_q(\phi_q,\chi_q^{-\{ak\}_{2n}})\right)\\
		&=(-1)^{\# X_2(a)+\# X_3(a)}\prod_{1\le j\le 4}\prod_{k\in U_j(a)}\left(J_q(\phi_q,\chi_q^{k})+J_q(\phi_q,\chi_q^{-k})\right)\\
		&=(-1)^{\# X_2(a)+\# X_3(a)}\prod_{0<k<n/2}\left(J_q(\phi_q,\chi_q^{k})+J_q(\phi_q,\chi_q^{-k})\right)\\
		&=(-1)^{\# X_2(a)+\# X_3(a)}\cdot R_q(\chi_q).
	\end{align*}
	By Lemma \ref{Lem. key lemma for q=3 mod 4}, we see that  
	$$\# X_2(a)+\# X_3(a)=\#\left\{k\in\left(0,\frac{n}{2}\right):\ \frac{n}{2}<\{ak\}_{2n}<\frac{3n}{2}\right\}\equiv 0\pmod{2\mathbb{Z}}.$$
	Thus, by the above results, we immediately obtain 
	$$\sigma_a(R_q(\chi_q))=R_q(\chi_q)$$
	for any $\sigma_a\in\Gal(\mathbb{Q}(\zeta_{q-1})/\mathbb{Q})$. By the Galois correspondence, we have $R_q(\chi_q)\in\mathbb{Q}$. Also, it is known that 
	$$\{\chi_q^s:\ 1\le s\le q-1\ \text{and}\ \gcd(s,q-1)=1\}$$
	is exactly the set of all generators of $\widehat{\mathbb{F}_q^{\times}}$. Hence, for any generator $\chi_q^s$ of $\widehat{\mathbb{F}_q^{\times}}$, where $1\le s\le q-1$ and $\gcd(s,q-1)=1$, we see that 
	$$R_q(\chi_q^s)=\prod_{0<k<n/2}\left(J_q(\phi_q,\chi_q^s)+J_q(\phi_q,\chi_q^{-s})\right)=\sigma_s(R_q(\chi_q))=R_q(\chi_q).$$
	This implies that $R_q(\chi_q)$ is independent of the choice of the generator $\chi_q$ of $\widehat{\mathbb{F}_q^{\times}}$. Moreover, for any integer $k\in (0,n/2)$, by Lemma \ref{Lem. eigenvalues of Aq} we see that 
	$$\frac{(-1)^k}{2}\left(J_q(\phi_q,\chi_q^k)+J_q(\phi_q,\chi_q^{-k})\right)$$
	is an algebraic integer. Thus, $$R_q(\chi_q)/2^{\lfloor n/2\rfloor}=R_q(\chi_q)/2^{(n-1)/2}\in\mathbb{Z}[\zeta_{q-1}]\cap\mathbb{Q}=\mathbb{Z},$$
	that is, $x_q$ and $R_q(\chi_q)$ are both rational integers. 
	
	(ii) Now we consider the case $q\equiv 1\pmod 4$, i.e., $n\equiv 0\pmod 2$.
	
	Let notations be as above. Applying Lemma \ref{Lem. transform of Jacobi sums}, we obtain 
	\begin{align}\label{Eq. a in the proof of Thm. global result for q=1 mod 4}
		&\prod_{k\in X_2(a)}\left(J_q(\phi_q,\chi_q^{ak})+J_q(\phi_q,\chi_q^{-ak})\right)\notag\\
		=&\prod_{k\in X_2(a)}\left(J_q(\phi_q,\chi_q^{\{ak\}_{2n}})+J_q(\phi_q,\chi_q^{-\{ak\}_{2n}})\right)\notag\\
		=&\prod_{k\in X_2(a)}\left(J_q(\phi_q,\chi_q^{n-\{ak\}_{2n}})+J_q(\phi_q,\chi_q^{-n+\{ak\}_{2n}})\right),
	\end{align}
	\begin{align}\label{Eq. b in the proof of Thm. global result for q=1 mod 4}
		&\prod_{k\in X_3(a)}\left(J_q(\phi_q,\chi_q^{ak})+J_q(\phi_q,\chi_q^{-ak})\right)\notag\\
		=&\prod_{k\in X_3(a)}\left(J_q(\phi_q,\chi_q^{\{ak\}_{2n}})+J_q(\phi_q,\chi_q^{-\{ak\}_{2n}})\right)\notag\\
		=&\prod_{k\in X_3(a)}\left(J_q(\phi_q,\chi_q^{-n+\{ak\}_{2n}})+J_q(\phi_q,\chi_q^{n-\{ak\}_{2n}})\right),
	\end{align}
	and
	\begin{align}\label{Eq. c in the proof of global result for q=1 mod 4}
		&\prod_{k\in X_4(a)}\left(J_q(\phi_q,\chi_q^{ak})+J_q(\phi_q,\chi_q^{-ak})\right)\notag\\
		=&\prod_{k\in X_4(a)}\left(J_q(\phi_q,\chi_q^{\{ak\}_{2n}})+J_q(\phi_q,\chi_q^{-\{ak\}_{2n}})\right)\notag\\
		=&\prod_{k\in X_4(a)}\left(J_q(\phi_q,\chi_q^{2n-\{ak\}_{2n}})+J_q(\phi_q,\chi_q^{-2n+\{ak\}_{2n}})\right).
	\end{align}
	
	Note again that $ak_1\not\equiv \pm ak_2 \pmod{n\mathbb{Z}}$ for any $0<k_1\neq k_2<n/2$. Combining this with (\ref{Eq. a in the proof of Thm. global result for q=1 mod 4})--(\ref{Eq. c in the proof of global result for q=1 mod 4}), for any $\sigma_a\in\Gal(\mathbb{Q}(\zeta_{q-1})/\mathbb{Q})$, one can verify that 
	\begin{align*}
		\sigma_a(R_q(\chi_q))
		&=\prod_{0<k<n/2}\left(J_q(\phi_q,\chi_q^{ak})+J_q(\phi_q,\chi_q^{-ak})\right)\\
		&=\prod_{1\le j\le 4}\prod_{k\in X_j(a)}\left(J_q(\phi_q,\chi_q^{\{ak\}_{2n}})+J_q(\phi_q,\chi_q^{-\{ak\}_{2n}})\right)\\
		&=\prod_{1\le j\le 4}\prod_{k\in U_j(a)}\left(J_q(\phi_q,\chi_q^{k})+J_q(\phi_q,\chi_q^{-k})\right)\\
		&=\prod_{0<k<n/2}\left(J_q(\phi_q,\chi_q^{k})+J_q(\phi_q,\chi_q^{-k})\right)\\
		&=R_q(\chi_q)
	\end{align*}
	
	Hence, $R_q(\chi_q)\in\mathbb{Q}$ by the Galois correspondence. By the same method used in the case $q\equiv 3\pmod 4$, it is easy to see that $R_q(\chi_q)/2^{\lfloor n/2\rfloor}=R_q(\chi_q)/2^{(n-2)/2}\in\mathbb{Z}$, and that $R_q(\chi_q)$ is independent of the choice of the generator $\chi_q$ of $\widehat{\mathbb{F}_q^{\times}}$.
	
	In view of the above, we have completed the proof.\qed 
	
	\section{Proof of Theorem \ref{Thm. xq and det Aq}}
	\setcounter{lemma}{0}
	\setcounter{theorem}{0}
	\setcounter{equation}{0}
	\setcounter{conjecture}{0}
	\setcounter{remark}{0}
	\setcounter{corollary}{0}
	
	{\noindent\bfseries Proof of Theorem \ref{Thm. xq and det Aq}}. Recall that $q=2n+1$ be an odd prime power. Let $\chi_q$ be a generator of $\widehat{\mathbb{F}_q^{\times}}$. 
	
	(i) We first consider the case $q=2n+1\equiv 3\pmod 4$, i.e., $n\equiv 1\pmod 2$. 
	
	Note that $J_q(\phi_q,\chi_q^n)+J_q(\phi_q,\chi_q^{-n})=2$ when $q\equiv3\pmod 4$. Since $n$ is odd, applying Lemma \ref{Lem. eigenvalues of Aq} and Lemma \ref{Lem. transform of Jacobi sums}, one can verify that 
	\begin{align*}
		\det A_q
		&=\prod_{1\le k\le n}\frac{(-1)^k}{2}\left(J_q(\phi_q,\chi_q^k)+J_q(\phi_q,\chi_q^{-k})\right)\\
		&=\frac{(-1)^{n(n+1)/2}}{2^{n-1}}\cdot R_q(\chi_q)\cdot\prod_{n/2<k<n}\left(J_q(\phi_q,\chi_q^k)+J_q(\phi_q,\chi_q^{-k})\right)\\
		&=\frac{(-1)^{(n+1)/2}}{2^{n-1}}\cdot R_q(\chi_q)\cdot \prod_{0<k<n/2}\left(J_q(\phi_q,\chi_q^{n-k})+J_q(\phi_q,\chi_q^{-n+k})\right)\\
		&=-\frac{1}{2^{n-1}}\cdot R_q(\chi_q)^2\\
		&=-x_q^2,
	\end{align*}
	where $x_q=R_q/2^{\lfloor n/2\rfloor}\in\mathbb{Z}$. 
	
	(ii) Next we focus on the case $q\equiv 1\pmod 4$, i.e., $n\equiv 0\pmod 2$. 
	
	Observe that $J_q(\phi_q,\chi_q^n)+J_q(\phi_q,\chi_q^{-n})=-2$ when $q\equiv1\pmod 4$. In addition, for $n/2$, by Lemma \ref{Lem. eigenvalues of Aq} we have 
	\begin{align*}
	 \frac{(-1)^{n/2}}{2}\left(J_q(\phi_q,\chi_q^{n/2})+J_q(\phi_q,\chi_q^{-n/2})\right)
&=\sum_{x\in\mathcal{S}_q}\phi_q(1+x)\chi_q^{n/2}(x)\\
&=\frac{1}{2}\sum_{x\in\mathbb{F}_q}\phi_q(1+x^2)\chi_q^{n/2}(x^2)\\
&=\frac{1}{2}\sum_{x\in\mathbb{F}_q}\phi_q(x^3+x)\\
&=-\frac{1}{2}a_q(E),
	\end{align*}
where $a_q(E)$ is defined by (\ref{Eq. definition of trace of Frobenius}). Combining the above results with Lemma \ref{Lem. eigenvalues of Aq} and Lemma \ref{Lem. transform of Jacobi sums}, one can verify that 
\begin{align*}
	\det A_q
&=\prod_{1\le k\le n}\frac{(-1)^k}{2}\left(J_q(\phi_q,\chi_q^k)+J_q(\phi_q,\chi_q^{-k})\right)\\
&=\frac{(-1)^{n(n+1)/2}}{2^n}\cdot (-2)\cdot\prod_{0<k<n}\left(J_q(\phi_q,\chi_q^k)+J_q(\phi_q,\chi_q^{-k})\right)\\
&=\frac{1}{2^{n-2}}\cdot \frac{1}{2}a_q(E)\cdot\prod_{k\in (0, n)\setminus\{n/2\}}\left(J_q(\phi_q,\chi_q^k)+J_q(\phi_q,\chi_q^{-k})\right)\\
&=\frac{1}{2^{n-2}}\cdot \frac{1}{2}a_q(E)\cdot R_q(\chi_q)\cdot \prod_{k\in (0, n/2)}\left(J_q(\phi_q,\chi_q^{n-k})+J_q(\phi_q,\chi_q^{-n+k})\right)\\
&=\frac{1}{2^{n-2}}\cdot \frac{1}{2}a_q(E)\cdot R_q(\chi_q)^2\\
&=\frac{1}{2}a_q(E)\cdot x_q^2.
\end{align*}
	
(iii) Finally, suppose that $q=p\equiv 1\pmod 4$ is a prime with $p=c_p^2+4d_p^2$ with $c_p,d_p\in\mathbb{Z}$ and $c_p\equiv 1\pmod 4$. Then it is known that (cf. \cite[Theorem 6.2.9]{BEW}) 
	$$\sum_{0<x<p/2}\phi_p(x^3+x)=-c_p.$$
Since $p\equiv 1\pmod 4$, the above equality implies that 
$$\frac{1}{2}a_p(E)=-\frac{1}{2}\sum_{x\in\mathbb{F}_p}\phi_p(x^3+x)=-\sum_{0<x<p/2}\phi_p(x^3+x)=c_p.$$
As $c_p\neq 0$, applying the above result to $\det A_p$, we obtain 
$$\frac{\det A_p}{c_p}=x_p^2.$$
	
In view of the above, we  have completed the proof. \qed 
	
\section{Proof of Theorem \ref{Thm. local properties of Rq}}
\setcounter{lemma}{0}
\setcounter{theorem}{0}
\setcounter{equation}{0}
\setcounter{conjecture}{0}
\setcounter{remark}{0}
\setcounter{corollary}{0}
	
Recall that $q=p^f$ is an odd prime power. Throughout this section, we always view the Jacobi sum $J_q(\chi_q^i,\chi_q^j)$ as an element of the local field $\mathbb{Q}_p(\zeta_{q-1})$ for any $i,j\in\mathbb{Z}$. 

In this section, we study the local properties of $R_q$. We first introduce the Teich\"{u}muller character $\omega_q$ of $\mathbb{F}_q$. Let $\zeta_{q-1}\in\mathbb{Q}_p^{\alg}$ be a primitive $(q-1)$-th root of unity. It is known that the extension $\mathbb{Q}_p(\zeta_{q-1})/\mathbb{Q}_p$ is unramified with $[\mathbb{Q}_p(\zeta_{q-1}):\mathbb{Q}_p]=f$. Also, the integral closure of $\mathbb{Z}_p$ in $\mathbb{Q}_p(\zeta_{q-1})$ is $\mathbb{Z}_p[\zeta_{q-1}]$. Thus, 
$$\mathbb{Z}_p[\zeta_{q-1}]/\mathfrak{p}\cong\mathbb{F}_q,$$
where $\mathfrak{p}=p\mathbb{Z}_p[\zeta_{q-1}]$ is a prime ideal. Throughout this section, we identify $\mathbb{F}_q$ with $\mathbb{Z}_p[\zeta_{q-1}]/\mathfrak{p}$. The Teich\"{u}muller character $\omega_q: \mathbb{F}_q\rightarrow \mathbb{Q}_p(\zeta_{q-1})$ is a multiplicative character of $\mathbb{F}_q$ defined by 
$$\omega_q(x\mod{\mathfrak{p}})\equiv x\pmod{\mathfrak{p}}$$
for any $x\in \mathbb{Z}_p[\zeta_{q-1}]$. Let $\mu_{q-1}(\mathbb{Q}_p^{\alg})=\{\zeta_{q-1}^k:\ k\in\mathbb{Z}\}$. Noting that 
$$\mathbb{F}_q^{\times}=\mu_{q-1}(\mathbb{F}_q^{\alg})=\left\{x\in\mathbb{F}_q^{\alg}:\ x^{q-1}=1\right\},$$
one can verify that the map $\pi_{\mathfrak{p}}: \mu_{q-1}(\mathbb{Q}_p^{\alg})\rightarrow \mathbb{F}_q^{\times}$, defined by $\pi_{\mathfrak{p}}(x)=x\mod{\mathfrak{p}}$, is a group isomorphism with 
$$\pi_{\mathfrak{p}}\circ \omega_q=\id_{\mathbb{F}_q^{\times}}.$$
Thus, $\omega_q$ is an isomorphism from $\mathbb{F}_q^{\times}$ to $\mu_{q-1}(\mathbb{Q}_p^{\alg})$. This implies that $\omega_q$ is a generator of $\widehat{\mathbb{F}_q^{\times}}$. 
	
Now we state our first lemma (cf. \cite[Proposition 3.6.4]{Cohen}).

\begin{lemma}\label{Lem. Cohen}
	Let notations be as above. Then for any integer $1\le i,j\le q-2$, we have 
	$$J_q(\omega_q^{-i},\omega_q^{-j})\equiv -\binom{i+j}{i}\pmod{\mathfrak{p}}.$$
	Moreover, if $i+j\ge q$, then 
	$$J_q(\omega_q^{-i},\omega_q^{-j})\equiv 0\pmod{\mathfrak{p}}.$$
\end{lemma}

We also need the following result (see \cite[Theorem 2.1.4]{BEW}).

\begin{lemma}\label{Lem. BEW}
	Let notations be as above. For any integer $k$ with $k\not\equiv 0\pmod{q-1}$, we have 
	$$J_q(\phi_q,\omega_q^{-k})=\omega_q^{-k}(4)\cdot J_q(\omega_q^{-k},\omega_q^{-k}).$$
\end{lemma}
	
Wu and Pan \cite[Lemma 4.1]{Wu and Pan} obtained the following result.

\begin{lemma}\label{Lem. Legendre by Wu and Pan}
	Let $p\equiv 1\pmod 4$ be a prime and let 
	$$B_p=\prod_{0<k<p/4}\binom{2k}{k}.$$
	Then 
	$$\left(\frac{B_p}{p}\right)=\left(\frac{2}{p}\right)=\left(\frac{\frac{p-1}{2}!}{p}\right).$$
\end{lemma}

Mordell \cite{Mordell} obtained the following classical result in 1961.

\begin{lemma}\label{Lem. Mordell}
	Let $p\equiv 3\pmod 4$ be a prime with $p>3$. Then 
	$$\frac{p-1}{2}!\equiv (-1)^{\frac{h(-p)+1}{2}} \pmod {p},$$
	where $h(-p)$ is the class number of $\mathbb{Q}(\sqrt{-p})$. 
\end{lemma}

The next lemma will play a key role in the proof of our second theorem.

\begin{lemma}\label{Lem. Legendre symbol when q=3 mod 4}
	Let $p\equiv 3\pmod 4$ be a prime with $p>3$, and let 
	$$y_p=\#\left\{k\in \left(0,\frac{p}{8}\right):\ \left(\frac{k}{p}\right)=-1\right\},\ \text{and}\ z_p=\#\left\{k\in \left(\frac{p}{8},\frac{p}{4}\right):\ \left(\frac{k}{p}\right)=1\right\}.$$
	Then 
	$$y_p+z_p\equiv 
	\begin{cases}
		0 \pmod{2}                            &  \mbox{if}\ p\equiv 3\pmod 8,\\
		\frac{1+h(-p)}{2} \pmod 2 &  \mbox{if}\ p\equiv 7\pmod 8.
	\end{cases}$$
\end{lemma}
	
\begin{proof}
	We first evaluate the sum 
	$$s_l=\sum_{0<k<p/4}\left(\frac{k}{p}\right).$$
	Let 
	$$s_r=\sum_{p/4<k<p/2}\left(\frac{k}{p}\right).$$
	Then by the Dirichlet class number formula, we have 
	\begin{equation}\label{Eq. a in the proof of Lem. Legendre symbol when q=3 mod 4}
		s_l+s_r=s_w=\sum_{0<k<p/2}\left(\frac{k}{p}\right)=\left(2-\left(\frac{2}{p}\right)\right)\cdot h(-p).
	\end{equation}
	On the other hand, since $p\equiv 3\pmod 4$, one can verify that 
	\begin{align*}
		s_w
		&=\sum_{k\in(0, p/2)\cap 2\mathbb{Z}}\left(\frac{k}{p}\right)+\sum_{k\in(0, p/2)\cap (2\mathbb{Z}+1)}\left(\frac{k}{p}\right)\\
		&=\sum_{0<k<p/4}\left(\frac{2k}{p}\right)+\sum_{k\in(p/2, p)\cap 2\mathbb{Z}}\left(\frac{p-k}{p}\right)\\
		&=\sum_{0<k<p/4}\left(\frac{2k}{p}\right)-\sum_{k\in(p/2, p)\cap 2\mathbb{Z}}\left(\frac{k}{p}\right)\\
		&=\sum_{0<k<p/4}\left(\frac{2k}{p}\right)-\sum_{p/4<k<p/2}\left(\frac{2k}{p}\right)\\
		&=\left(\frac{2}{p}\right)\cdot (s_l-s_r).
	\end{align*}
	Combining this with (\ref{Eq. a in the proof of Lem. Legendre symbol when q=3 mod 4}), we obtain 
	\begin{equation}\label{Eq. b in the proof of Lem. Legendre symbol when q=3 mod 4}
		s_l=\frac{1}{2}\left(1+\left(\frac{2}{p}\right)\right)s_w=
		\begin{cases}
			0          &  \mbox{if}\ p\equiv 3\pmod 8,\\
			h(-p)   &  \mbox{if}\ p\equiv 7\pmod 8.
		\end{cases}
	\end{equation}
	
	Next we compute 
	$$r_l=\#\left\{k\in\left(0,\frac{p}{4}\right):\ \left(\frac{k}{p}\right)=1\right\}.$$
	Let 
	$$n_l=\#\left\{k\in\left(0,\frac{p}{4}\right):\ \left(\frac{k}{p}\right)=-1\right\}.$$
	Then 
	\begin{align*}
		r_l+n_l&=\frac{p-3}{4},\\
		r_l-n_l&=\sum_{0<k<p/4}\left(\frac{x}{p}\right)=s_l.
	\end{align*}
	From this and (\ref{Eq. b in the proof of Lem. Legendre symbol when q=3 mod 4}), we obtain 
	\begin{equation}\label{Eq. c in the proof of Lem. Legendre symbol when q=3 mod 4}
		r_l=\frac{p-3+4s_l}{8}=
		\begin{cases}
			\frac{p-3}{8}                   &  \mbox{if}\ p\equiv 3\pmod 8,\\
			\frac{p-3+4h(-p)}{8}    &  \mbox{if}\ p\equiv 7\pmod 8.
		\end{cases}
	\end{equation}
	
	Finally, we determine $y_p+z_p\mod{2\mathbb{Z}}$. By the definitions of $y_p$ and $z_p$, we obtain 
	$$\left\lfloor\frac{p}{8}\right\rfloor-y_p+z_p=r_l.$$
	Assembling this and (\ref{Eq. c in the proof of Lem. Legendre symbol when q=3 mod 4}) gives 
	$$y_p+z_p\equiv \left\lfloor\frac{p}{8}\right\rfloor+r_l\equiv 
	\begin{cases}
		0 \pmod{2}                             &  \mbox{if}\ p\equiv 3\pmod 8,\\
		\frac{1+h(-p)}{2}\pmod{2} &  \mbox{if}\ p\equiv 7\pmod 8.
	\end{cases}$$
	
	In view of the above, we have completed the proof.
\end{proof}
	
	Now we are in a position to prove our second theorem.
	
	{\noindent\bfseries Proof of Theorem \ref{Thm. local properties of Rq}}. Recall that $q=p^f=2n+1$ is an odd prime power, and that the Teich\"{u}muller character $\omega_q$ of $\mathbb{F}_q$ is a generator of $\widehat{\mathbb{F}_q^{\times}}$. By Theorem \ref{Thm. Rq is an integer}, the integer $R_q(\chi_q)$ is independent of the choice of the generator $\chi_q$ of $\widehat{\mathbb{F}_q^{\times}}$. Hence, throughout this proof, we set $R_q=R_q(\omega_q)$. 
	
	(i) We first consider the case $f>1$ and $q\neq 9$. 
	
	If $q\equiv 3\pmod 4$, then clearly $f\ge 3$. Hence $q\ge 9p>2p+3$. For $q\equiv 1\pmod 4$, since $q\neq 9$, we have $q=p^f\ge 5p>2p+3$. By the above results, we have 
	$$\frac{p+1}{2}<\frac{n}{2}.$$
	Note that 
	$$n+\frac{p+1}{2}=0\cdot 1+\frac{p+1}{2}\cdot p+\cdots+\frac{p-1}{2}\cdot p^{f-1}.$$
	Using the Lucas congruence, we obtain 
	\begin{equation}\label{Eq. a in the proof of Thm. local result}
		\binom{n+(p+1)/2}{(p+1)/2}\equiv \binom{0}{(p+1)/2}\binom{(p+1)/2}{0}\cdots\binom{(p-1)/2}{0}\equiv 0\pmod {p}.
	\end{equation}
	
	On the other hand, as $\omega_q^{\pm n}=\phi_q$, by Lemma \ref{Lem. transform of Jacobi sums} we obtain 
	\begin{align*}
		R_q
		&=\prod_{0<k<n/2}\left(J_q(\phi_q,\omega_q^k)+J_q(\phi_q,\omega_q^{-k})\right)\\
		&=\prod_{0<k<n/2}\left(\phi_q(-1)J_q(\omega_q^{-n},\omega_q^{-n-k})+J_q(\omega_q^{-n},\omega_q^{-k})\right).
	\end{align*}
	Applying Lemma \ref{Lem. Cohen} and noting that $n+(n+k)\ge q$ for any $k\in (0, n/2)$, we obtain 
	$$R_q\equiv \prod_{0<k<n/2}-\binom{n+k}{k}\equiv -\binom{n+(p+1)/2}{(p+1)/2}\prod_{k\in(0, n/2)\setminus\{(p+1)/2\}}-\binom{n+k}{k}\pmod{\mathfrak{p}}.$$
	Using (\ref{Eq. a in the proof of Thm. local result}) and noting that $R_q\in\mathbb{Z}$, we have 
	$$R_q\equiv 0\pmod{p}$$
	whenever $f>1$ and $q\neq 9$. 
	
	(ii) Next we consider the case $q=9$. 
	
	Since $R_9=J_9(\phi_9,\omega_9)+J_9(\phi_9,\omega_9^{-1})$, with the help of a computer, we obtain $R_9=-2$. 
	
	(iii) Now we focus on the case $f=1$, i.e., $q=p$. 
	
	In this case, $\mathbb{Q}_p(\zeta_{p-1})=\mathbb{Q}_p$ and $\mathfrak{p}=p\mathbb{Z}_p$. The desired results hold trivially when $p=3$. From now on, we suppose $p>3$. Assembling Lemma \ref{Lem. transform of Jacobi sums} and Lemmas \ref{Lem. Cohen}--\ref{Lem. BEW} gives
	\begin{align}\label{Eq. b in the proof of Thm. local result}
		R_p
		&=\prod_{0<k<n/2}\left(J_p(\phi_p,\omega_p^k)+J_q(\phi_p,\omega_p^{-k})\right)\notag\\
		&=\prod_{0<k<n/2}\left(\phi_p(-1)J_p(\omega_p^{-n},\omega_p^{-n-k})+J_p(\omega_p^{-n},\omega_p^{-k})\right)\notag\\
		&=\prod_{0<k<n/2}\left(\phi_p(-1)J_p(\omega_p^{-n},\omega_p^{-n-k})+\omega_p^{-k}(4)J_p(\omega_p^{-k},\omega_p^{-k})\right)\notag\\
		&\equiv \prod_{0<k<n/2}\omega_p^{-k}(4)J_p(\omega_p^{-k},\omega_p^{-k})\notag\\
		&\equiv (-1)^{\lfloor n/2 \rfloor}\left(\frac{1}{4}\right)^{\lfloor n/2 \rfloor(1+\lfloor n/2 \rfloor)/2}\prod_{0<k<n/2}\binom{2k}{k}\pmod{p\mathbb{Z}_p}.
	\end{align}
	
	(iv) Finally, we concentrate on $(\frac{R_p}{p})$. We will complete the remaining proof by considering two cases.
	
	{\bf Case 1}. $p\equiv 3\pmod 4$ and $p>3$. 
	
	We first compute the product 
	$$t_w=\prod_{0<k<p/2}\left(\frac{k}{p}\right)^{\lfloor\frac{k-1}{2}\rfloor}.$$
	For $j=0,3$, set 
	$$n_j=\#\left\{k\in\left(0,\frac{p}{2}\right): \left(\frac{k}{p}\right)=-1\ \text{and}\ k\equiv j\pmod 4\right\}.$$
	Since $p\equiv 3\pmod 4$, one can verify that 
	\begin{equation*}
		n_0=\#\left\{k\in\left(0,\frac{p}{8}\right): \left(\frac{k}{p}\right)=-1\right\}=y_p,
	\end{equation*}
	and that 
	\begin{align*}
		n_3
		&=\#\left\{k\in\left(\frac{p}{2},p\right):\ \left(\frac{p-k}{p}\right)=-1\ \text{and}\ k\equiv 0\pmod 4\right\}\\
		&=\#\left\{k\in\left(\frac{p}{2},p\right):\ \left(\frac{k}{p}\right)=1\ \text{and}\ k\equiv 0\pmod 4\right\}\\
		&=\#\left\{k\in\left(\frac{p}{8},\frac{p}{4}\right):\ \left(\frac{k}{p}\right)=1\right\}\\
		&=z_p,
	\end{align*}
	where $y_p$ and $z_p$ are defined by Lemma \ref{Lem. Legendre symbol when q=3 mod 4}. Applying Lemma \ref{Lem. Legendre symbol when q=3 mod 4} to $t_w$, we see that 
	\begin{equation}\label{Eq. c in the proof of Thm. local result for q=3 mod 4}
		t_w=(-1)^{n_0+n_3}=(-1)^{y_p+z_p}=
		\begin{cases}
			1                                             &  \mbox{if}\ p\equiv 3\pmod 8,\\
			(-1)^{\frac{h(-p)+1}{2}}  &  \mbox{if}\ p\equiv 7\pmod 8.
		\end{cases}
	\end{equation}
	
	Next we focus on $(\frac{B_p}{p})$, where 
		$$B_p=\prod_{0<k<p/4}\binom{2k}{k}.$$
	By computations one can verify that 
	\begin{align*}
		\left(\frac{B_p}{p}\right)
		&=\prod_{1\le k\le (p-3)/4}\left(\frac{(2k)!}{p}\right)\\
		&=\prod_{1\le k\le (p-3)/2}\left(\frac{k}{p}\right)^{\frac{p-3}{4}}\left(\frac{k}{p}\right)^{\lfloor\frac{k-1}{2}\rfloor}\\
		&=\prod_{1\le k\le (p-1)/2}\left(\frac{k}{p}\right)^{\frac{p-3}{4}}\left(\frac{k}{p}\right)^{\lfloor\frac{k-1}{2}\rfloor}\\
		&=\left(\frac{\frac{p-1}{2}!}{p}\right)^{\frac{p-3}{4}}\cdot t_w.
	\end{align*}
	Using Lemma \ref{Lem. Mordell} and (\ref{Eq. c in the proof of Thm. local result for q=3 mod 4}), we obtain 
	\begin{equation*}
		\left(\frac{B_p}{p}\right)=(-1)^{\frac{h(-p)+1}{2}\cdot\frac{p-3}{4}}\cdot t_w=1.
	\end{equation*}
	From this and (\ref{Eq. b in the proof of Thm. local result}), we obtain 
	$$\left(\frac{R_p}{p}\right)=(-1)^{\frac{p-3}{4}}\left(\frac{B_p}{p}\right)=(-1)^{\frac{p-3}{4}}=\left(\frac{-2}{p}\right).$$
	
	{\bf Case 2}. $p=2n+1\equiv 1\pmod 4$. 
	
	Note that $(\frac{-1}{p})=1$ in this case. By (\ref{Eq. b in the proof of Thm. local result}) and Lemma \ref{Lem. Legendre by Wu and Pan}, we obtain 
	\begin{align*}
		\left(\frac{R_p}{p}\right)=\left(\frac{\prod_{0<k<n/2}\binom{2k}{k}}{p}\right)=\left(\frac{B_p}{p}\right)\left(\frac{\binom{n}{n/2}}{p}\right)=\left(\frac{B_p}{p}\right)\left(\frac{n!}{p}\right)=1,
	\end{align*}
	where $B_p$ is defined by Lemma \ref{Lem. Legendre by Wu and Pan}.
	
	In view of the above, we have completed the proof. \qed 
	
	\Ack\ This research was supported by the Natural Science Foundation of China (Grant No. 12101321) and the Natural Science Foundation of the Higher Education Institutions of Jiangsu Province (Grant No. 25KJB110010).

	\end{document}